\documentclass[11pt,leqno]{amsart}

\usepackage[latin1]{inputenc}
\usepackage{amsmath}
\usepackage{amsfonts}
\usepackage{amssymb}
\usepackage{amsthm}
\usepackage[english]{babel}
\usepackage[T1]{fontenc}
\usepackage{graphicx}

\usepackage[pdftex]{hyperref}

\theoremstyle{plain}

\newtheorem{thm}{Theorem}[section]
\newtheorem{main}{Theorem}
\newtheorem{lem}[thm]{Lemma}

\newcommand{\Q}{\mathbb{Q}}
\newcommand{\Z}{\mathbb{Z}}
\newcommand{\B}{\mathcal{B}}
\DeclareMathOperator{\End}{End}

\title[Positive laws on large sets of generators]{Positive laws on large sets of generators:
counterexamples for infinitely generated groups}
\author{Cristina Acciarri}

\address{\textnormal{Cristina Acciarri,}\\ Dipartimento di Matematica Pura ed Applicata\\
Universit\`a degli Studi  dell'Aquila\\
I-67010 Coppito, L'Aquila (Italy)}
\email{acciarricristina@yahoo.it}

\author{Gustavo A. Fern\'andez-Alcober}
\address{ \textnormal{Gustavo A. Fern\'andez-Alcober,} Matematika Saila\\ Euskal Herriko Unibertsitatea
\\ 48080 Bilbao (Spain)}
\email{gustavo.fernandez@ehu.es}



\begin{document}

\begin{abstract}
Shumyatsky and the second author proved that if $G$ is a finitely generated residually finite $p$-group satisfying a law, then, for almost all primes, the fact that a normal and commutator-closed set of generators satisfies a positive law implies that the whole of $G$ also satisfies a (possibly different) positive law.
In this paper, we construct a counterexample showing that the hypothesis of
finite generation of the group $G$ cannot be dispensed with.
\end{abstract}

\maketitle

\section{Introduction}

A group word is called \emph{positive} if it does not involve any inverses of
the variables.
If $\alpha$ and $\beta$ are two different positive words, a subset $T$ of a group $G$
is said to satisfy the \emph{positive law} $\alpha\equiv \beta$ if every substitution
of elements of $T$ for the variables gives the same value for $\alpha$ and for
$\beta$.
The \emph{degree} of the law is the maximum of the lengths of the words $\alpha$
and $\beta$.
A prominent positive law is the \emph{Malcev law} $M_c(x,y)$ given by the relation
$\alpha_c(x,y)\equiv \beta_c(x,y)$, where $\alpha_c$ and $\beta_c$ are defined
by $\alpha_0=x$, $\beta_0=y$, and the recursive relations
\[
\alpha_c=\alpha_{c-1}\beta_{c-1}
\quad
\text{and}
\quad
\beta_c=\beta_{c-1}\alpha_{c-1}.
\]
Thus $M_1(x,y)$ is the abelian law $xy\equiv yx$, and $M_2(x,y)$ is the law
$xyyx\equiv yxxy$.
Throughout this paper, when we speak about a Malcev law $M_c(x,y)$, we always
assume that $c\ge 1$.

Every nilpotent group of class $c$ satisfies the law $M_c(x,y)$, and an
extension of a nilpotent group of class $c$ by a group of finite exponent $e$
satisfies the positive law $M_c(x^e,y^e)$.
Malcev asked whether, conversely, a group which satisfies a positive law is
nilpotent-by-(finite exponent).
This question was answered in the negative by Olshanskii and Storozhev in
\cite{ols}.
However, the answer is positive for a large class of groups: Burns and
Medvedev proved in \cite{bur} that \emph{a locally graded group satisfying a
positive law is nilpotent-by-(locally finite of finite exponent)}.
(See also the paper \cite{baj} by Bajorska and Macedo\'{n}ska.)

An interesting question regarding positive laws is the following: under what
conditions does a positive law on a set $T$ of generators of a group $G$ imply
a (possibly different) positive law on the whole of $G$?
This problem is inspired by the following particular but important case: is it
true that a positive law on the set of all values of a word $w$ in a group $G$
implies a positive law on the verbal subgroup $w(G)$?
One of the conditions that must be certainly fulfilled in the first question is
that the set $T$ of generators has to be large in some sense.
For example, a free product $G=P*Q$ of two finite $p$-groups is generated by the
set $T=P\cup Q$, which satisfies a positive law of the form $x^{p^k}\equiv 1$,
but $G$ does not satisfy a positive law unless $|P|,|Q|\leq2$.
On the other hand, the set of values of a word is to some extent large; note that
it is a normal subset and, on occasions, also commutator-closed (i.e.\ closed
under taking commutators of its elements).
This happens, for example, with the simple commutators $[x_1,\ldots,x_m]$,
and with the derived words.

Shumyatsky and the second author \cite{fer} have considered the question of the previous
paragraph in the realm of finitely generated residually finite $p$-groups.
One of their main results is the following: \emph{for every $n$, there exists
a finite set $P(n)$ of primes such that, if $p\not\in P(n)$ and $G$ is a finitely
generated residually finite $p$-group which satisfies a law and which can be generated by a normal
and commutator-closed $T$ satisfying a positive law of degree $n$, then also $G$
satisfies a positive law.}
Thus `normal and commutator-closed' is a valid sense of largeness in the above setting
(for example, for soluble residually finite $p$-groups), a fact which can be applied to several
important instances of the problem for word values and verbal subgroups.

Our goal in this paper is to show that the hypothesis of finite generation of $G$
cannot be dispensed with in the previous result.
More precisely, we prove the following.

\begin{main}
\label{mainthm}
For every $c\ge 3$, there exists an infinitely generated metabelian group $G$
such that:
\begin{enumerate}
\item
$G$ is a residually finite $p$-group for all primes $p$.
\item
$G$ can be generated by a commutator-closed normal subset $T$ satisfying the positive law
$M_c(x,y)$.
\item
$G$ does not satisfy any positive laws.
\end{enumerate}
\end{main}

The main tool which is needed for the construction of this counterexample is to characterize
when a union of cosets of an abelian normal subgroup satisfies a Malcev law, provided that the
representatives of the cosets commute with each other.
This is the goal of Section 2.
Once this characterization is obtained, in Section 3 we proceed to construct the counterexample,
and prove that our main theorem, Theorem \ref{mainthm} holds.
It is noteworthy that the theory of monomial ideals in polynomial algebras plays an important
role in the proof.

\section{The Malcev law on unions of cosets of an abelian normal subgroup}

If $G$ is a group and $A$ is an abelian normal subgroup of $G$, then every element
$t\in G$ defines an automorphism of $A$ by conjugation, which we denote by the same
letter $t$.
Since the set $\End(A)$ of endomorphisms of $A$ is a ring, we can combine these
automorphisms with the operations of addition and composition (which we denote by
juxtaposition).

We begin by determining when two elements in cosets $tA$ and $uA$, with $t$ and $u$
commuting, satisfy a Malcev law.

\begin{lem}
\label{substitution}
Let $G$ be a group, and let $A$ be an abelian normal subgroup of $G$.
If $t,u\in G$ commute and $a,b\in A$, then the Malcev law $M_c(x,y)$
holds for the substitution $x=ta$, $y=ub$ if and only if
\[
a^{f_c(u,t)}=b^{f_c(t,u)},
\]
where
\[
f_c(X,Y)=(X-1)\prod_{i=0}^{c-2} \, (X^{2^i}Y^{2^i}-1).
\]
\end{lem}

\begin{proof}
We define, for every $c\ge 1$, the word $w_c(x,y)=\beta_c(x,y)^{-1}\alpha_c(x,y)$.
The lemma will be proved if we see that
\[
w_c(ta,ub) = a^{f_c(u,t)}b^{-f_c(t,u)}.
\]
We argue by induction on $c$.
If $c=1$, then
\[
w_1(ta,ub)
=
(ubta)^{-1}(taub)
=
(utb^ta)^{-1}(tua^ub)
=
a^{u-1}b^{1-t},
\]
and the result is true.
Assume now that $c>1$.
Since
\[
w_c
=
\beta_c^{-1}\alpha_c
=
\alpha_{c-1}^{-1}\beta_{c-1}^{-1}\alpha_{c-1}\beta_{c-1}
=
w_{c-1}^{\alpha_{c-1}} w_{c-1}^{-1},
\]
it follows from the induction hypothesis that
\begin{equation}
\label{w_c}
w_c(ta,ub)
=
(a^{f_{c-1}(u,t)}b^{-f_{c-1}(t,u)})^{\alpha_{c-1}(ta,ub)}
(a^{-f_{c-1}(u,t)}b^{f_{c-1}(t,u)}).
\end{equation}
Now, since $A$ is abelian, in order to calculate the conjugate
in this last expression, we only need to know the value of
$\alpha_{c-1}(ta,ub)$ modulo $A$.
Since $\alpha_{c-1}$ has weight $2^{c-2}$ in both $x$ and $y$,
and $t$ and $u$ commute, it follows that
\[
\alpha_{c-1}(ta,ub) \equiv \alpha_{c-1}(t,u)
\equiv t^{2^{c-2}}u^{2^{c-2}} \pmod A.
\]
By putting this value into (\ref{w_c}), we get
\[
w_c(ta,ub)
=
a^{f_{c-1}(u,t)(t^{2^{c-2}}u^{2^{c-2}}-1)}
b^{-f_{c-1}(t,u)(t^{2^{c-2}}u^{2^{c-2}}-1)},
\]
which concludes the proof.
\end{proof}

Now we characterize when the unions of cosets of $A$ that we are interested in
satisfy a Malcev law.

\begin{thm}
\label{malcev on union of cosets}
Let $G$ be a group, and let $A$ be an abelian normal subgroup of $G$.
Consider a union of cosets $T=t_1 A\cup\cdots\cup t_n A\cup A$, where
$t_1,\ldots,t_n$ commute with each other.
Suppose that $t_1,\ldots,t_n$, as endomorphisms of $A$, satisfy the
following conditions:
\begin{enumerate}
\item
$(t_i-1)^c=0$, for all $i=1,\ldots,n$.
\item
$(t_i-1)(t_it_j-1)^{c-1}=0$, for $1\le i\ne j\le n$.
\end{enumerate}
Then the subset $T$ satisfies $M_c(x,y)$.
Conversely, if $T$ satisfies $M_c(x,y)$, and if $G$ is nilpotent and
$A$ is torsion-free, then $t_1,\ldots,t_n$ satisfy conditions (i) and
(ii) above.
\end{thm}

\begin{proof}
The law $M_c(x,y)$ holds in the subset $T$ if and only if it holds for every
substitution $x=ta$, $y=ub$, where $t,u\in\{1,t_1,\ldots,t_n\}$ and $a,b\in A$.
By considering the case where $a=1$ and $b$ is arbitrary, it readily follows
from Lemma~\ref{substitution} that $T$ satisfies $M_c(x,y)$ if and only if
$f_c(t,u)$ annihilates $A$ for every $t,u\in\{1,t_1,\ldots,t_n\}$.
Put differently, a necessary and sufficient condition for $T$ to satisfy $M_c(x,y)$
is that the substitution $X\mapsto t_i$ in $f_c(X,1)$ and $f_c(X,X)$, and the substitution
$X\mapsto t_i$, $Y\mapsto t_j$ in $f_c(X,Y)$, with $i\ne j$, always induce the zero
endomorphism of $A$.

Since
\begin{equation}
\label{igualdad polinomios1}
f_c(X,1)
=
(X-1)^{c}\prod_{i=1}^{c-2} \, (X^{2^i-1}+\cdots+X+1),
\end{equation}
\begin{equation}
\label{igualdad polinomios1bis}
f_c(X,X)
=
f_c(X,1) \prod_{i=0}^{c-2} \, (X^{2^i}+1),
\end{equation}
and
\begin{equation}
\label{igualdad polinomios2}
f_c(X,Y)
=
(X-1)(XY-1)^{c-1}\prod_{i=1}^{c-2} \, ((XY)^{2^i-1}+\cdots+XY+1),
\end{equation}
it is clear that, if conditions (i) and (ii) of the statement hold, then
$T$ satisfies $M_c(x,y)$.
This proves the first assertion of the theorem.

Conversely, suppose now that $T$ satisfies $M_c(x,y)$, that $G$ is nilpotent
and that $A$ is torsion-free.
By (\ref{igualdad polinomios1}), (\ref{igualdad polinomios1bis}),
and (\ref{igualdad polinomios2}), we have
\begin{equation}
\label{igualdad polinomios3}
f_c(X,1)=(X-1)^cg_c(X),
\quad
f_c(X,X)=(X-1)^ch_c(X),
\end{equation}
and
\begin{equation}
\label{igualdad polinomios4}
f_c(X,Y)=(X-1)(XY-1)^{c-1}g_c(XY),
\end{equation}
for some polynomials $g_c(X),h_c(X)\in\Z[X]$ which are coprime to $X-1$.
Now we claim that, for every polynomial $j(X)\in\Z[X]$ which is coprime to $X-1$,
and for every automorphism $\varphi$ of $A$ which is induced by conjugation by
an element of $G$, the endomorphism $j(\varphi)$ is injective.
Once this is proved, it follows from (\ref{igualdad polinomios3}) and
(\ref{igualdad polinomios4}), and from the discussion in the first paragraph
of the proof, that (i) and (ii) must hold.

Hence it only remains to prove the claim.
Let $k$ be the nilpotency class of $G$.
Since $j(X)$ is coprime to $X-1$, by using B\'ezout's identity in $\Q[X]$ we get
an expression of the form
\begin{equation}
\label{bezout}
p(X)(X-1)^k+q(X)j(X)=m,
\end{equation}
where $p(X),q(X)\in\Z[X]$, and $m$ is a positive integer.
Now, since $G$ is nilpotent of class $k$ and $\varphi$ is induced by conjugation
by an element of $G$, we have $(\varphi-1)^k=0$.
By substituting $\varphi$ for $X$ in (\ref{bezout}), it follows that
$q(\varphi)j(\varphi)=m1_A$.
Taking into account that $A$ is torsion-free, we conclude that $j(\varphi)$ is
injective, as desired.
\end{proof}

\section{The construction of the counterexample}

The key to our counterexample is the next lemma, where we show that for every $n\ge c$
there exists a nilpotent group $G_n$ which can be generated by a normal and commutator-closed
subset $T_n$ satisfying $M_c(x,y)$, but nevertheless $G_n$ does not satisfy any law $M_k(x,y)$
for $k\le n$.
Thus the `distance' between the Malcev laws satisfied by $T_n$ and $G_n$ increases as $n$ goes
to infinity.

\begin{lem}
\label{Gn}
Let $c\ge 3$ be a fixed integer.
Then, for every $n\ge c$ there exists a finitely generated nilpotent torsion-free group
$G_n=B_n\ltimes A_n$ satisfying the following properties:
\begin{enumerate}
\item
$A_n$ and $B_n$ are abelian groups.
Thus $G_n$ is metabelian.
\item
$B_n$ can be generated by $n$ elements $t_1,\ldots,t_n$ such that the subset
$T_n=t_1A_n\cup \cdots \cup t_nA_n\cup A_n$ satisfies the law $M_c(x,y)$.
\item
$G_n$ does not satisfy $M_n(x,y)$.
More precisely, for every $e\ge 1$, the law $M_n(x,y)$ is not satisfied in the coset
$(t_1\ldots t_n)^eA_n$.
\end{enumerate}
\end{lem}

\begin{proof}
The idea of the proof is to put $A_n=\Z^d$ for some $d$ (to be determined in the
course of the proof), and to let $t_1,\ldots,t_n$ be commuting matrices in $GL_d(\Z)$
which fulfil the necessary conditions for $T_n$ to satisfy $M_c(x,y)$, and for $G_n$
not to satisfy $M_n(x,y)$.
These are the conditions that can be read in Theorem~\ref{malcev on union of cosets}.
The matrices $t_1,\ldots,t_n$ will arise from the regular representation of an
appropriate quotient of the algebra of polynomials $\Q[X_1,\ldots,X_n]$.

Consider the ideal
\[
\mathfrak{a}
=
((X_1-1)^{c},\ldots,(X_n-1)^{c}, (X_i-1)(X_iX_j-1)^{c-1}
\mid 1\le i\ne j\le n)
\]
of $\Q[X_1,\ldots,X_n]$.
Under the isomorphism $X_i\mapsto X_i+1$, this ideal maps onto
\[
\mathfrak{b}
=
(X_1^c,\ldots,X_n^c,X_i(X_i+X_j+X_iX_j)^{c-1}
\mid 1\le i\ne j\le n).
\]
One can readily check that $\mathfrak{b}$ is contained in the monomial ideal
\[
\mathfrak{c}
=
(X_1^{i_1}\ldots X_n^{i_n} \mid
\text{$i_1+\cdots+i_n=c$ and $i_j\ge 2$ for some $j$});
\]
note that $c\ge 3$ is necessary for this.
Also, if $\mathfrak{m}$ is the maximal ideal of $\Q[X_1,\ldots,X_n]$ generated by
all the indeterminates $X_1,\ldots,X_n$, then $\mathfrak{m}^{n+1}$ is contained
in $\mathfrak{c}$, since $n\ge c$.

Let $e\ge 1$ be an arbitrary integer.
Since
\[
(X_1+1)^e(X_2+1)^e\ldots (X_n+1)^e-1
\equiv
e(X_1+X_2+\cdots+X_n)
\pmod{\mathfrak{m}^2},
\]
it follows that
\[
((X_1+1)^e(X_2+1)^e\ldots (X_n+1)^e-1)^n
\equiv
e^n(X_1+X_2+\cdots+X_n)^n
\pmod{\mathfrak{m}^{n+1}}.
\]
This last congruence also holds modulo $\mathfrak{c}$, since
$\mathfrak{m}^{n+1}\subseteq \mathfrak{c}$.
As a consequence, we have
\begin{equation}
\label{congruence}
((X_1+1)^e(X_2+1)^e\ldots (X_n+1)^e-1)^n
\equiv
e^n n! X_1\ldots X_n
\pmod{\mathfrak{c}}.
\end{equation}
On the other hand, by \cite[Lemma 2, page 67]{cox}, we have
\[
X_1\ldots X_n \not\in \mathfrak{c},
\]
since $\mathfrak{c}$ is a monomial ideal and $X_1\ldots X_n$ is not divisible
by any of the generators in the definition of $\mathfrak{c}$.
Thus, it follows from (\ref{congruence}) that
\begin{equation}
\label{not in c}
((X_{1}+1)^e(X_{2}+1)^e\ldots(X_{n}+1)^e-1)^n
\not\in \mathfrak{c}.
\end{equation}

Now, put $A=\Q[X_1,\ldots,X_n]/\mathfrak{c}$, and let $d$ be the dimension
of $A$ as a $\Q$-vector space.
The set
\[
\B
=
\{ X_1^{i_1}\ldots X_n^{i_n}+\mathfrak{c}
\mid
\text{$X_1^{i_1}\ldots X_n^{i_n}$ is not a multiple of a generator of $\mathfrak{c}$}
\}
\]
is a basis of $A$, by \cite[Proposition 4, page 229]{cox}.
We order $\B$ first by total degree of the monomials, and
then arbitrarily among monomials of the same degree.
Let us consider the regular representation $\varphi$ of $A$ in $M_d(\Q)$,
where matrices are taken with respect to the basis $\B$, and put
$t_i=\varphi(X_i+1+\mathfrak{c})$.
Obviously, $t_1,\ldots,t_n$ commute with each other.
Also, since the basis $\B$ consists only of monomials, and these are ordered
according to their degree, the matrices $t_i$ have only $0$ and $1$ entries,
and are upper unitriangular.
In other words, $t_i\in UT_d(\Z)$, the group of upper unitriangular matrices
over the integers.

Hence, we can consider the semidirect product $G_n=B_n\ltimes A_n$ of
the groups $A_n=\Z^d$ and $B_n=\langle t_1,\ldots,t_n \rangle$,
with respect to the natural action of $B_n$ on $A_n$.
Clearly, $G_n$ satisfies (i).
Since $UT_d(\Z)$ is a torsion-free group (see \cite[page 128]{rob}), also $B_n$ is
torsion-free.
Hence the same is true for $G_n$.
On the other hand, since $A_n$ and $B_n$ are abelian, we have
$\gamma_i(G_n)=[A_n,B_n,\overset{i-1}{\ldots},B_n]$ for all $i\ge 1$
(see Lemma 15.2 in Chapter 3 of \cite{hup}).
Since $B_n$ is contained in the unitriangular group, it follows that
$G_n$ is a nilpotent group.

On the other hand, since $X_i^c$ and $X_i(X_i+X_j+X_iX_j)^{c-1}$ lie in
$\mathfrak{c}$, it readily follows that $(t_i-1)^c=(t_i-1)(t_it_j-1)^{c-1}=0$
for all $1\le i\ne j\le n$.
Also, as a consequence of (\ref{not in c}), we have
$(t_1^e\ldots t_n^e-1)^n\ne 0$ for every $e\ge 1$.
Thus we can conclude from Theorem~\ref{malcev on union of cosets} that
$T_n$ satisfies $M_c(x,y)$, and that $M_n(x,y)$ is never satisfied in a coset
of the form $t_1^e\ldots t_n^eA_n$, with $e\ge 1$. 
\end{proof}

We are now ready to prove our main theorem, Theorem \ref{mainthm}.

\begin{thm}
For every $c\ge 3$, there exists an infinitely generated metabelian group $G$
such that:
\begin{enumerate}
\item
$G$ is a residually finite $p$-group for all primes $p$.
\item
$G$ can be generated by a commutator-closed normal subset $T$ satisfying the positive law
$M_c(x,y)$.
\item
$G$ does not satisfy any positive laws.
\end{enumerate}
\end{thm}

\begin{proof}
In the proof, we use the same notation as in Lemma~\ref{Gn}.
We define $G$ to be the restricted direct product $\prod_{n\ge c}\, G_n$.
Note that $G$ is metabelian.
Since $G_n$ is a finitely generated nilpotent torsion-free group, it is a residually finite $p$-group 
for all primes $p$, by a result of Gruenberg \cite{gru}.
As a consequence, the same is true for $G$ and (i) holds.

Since the direct product $G$ is restricted and $G_n=\langle T_n \rangle$ for all $n$,
it follows that the subset $T=\cup_{n\ge c} \, T_n$ generates $G$.
By the definition of $T_n$, it is clear that it is a normal subset of $G_n$, and also
commutator-closed.
(Recall that $t_1,\ldots,t_n$ commute with each other.)
As a consequence, $T$ is commutator-closed and a normal subset of $G$.
Also, since every $T_n$ satisfies the law $M_c(x,y)$, also does $T$: note that two
elements from $T_n$ and $T_m$, with $n\ne m$, commute.
Thus we obtain (ii).

Finally, let us see that $G$ cannot satisfy a positive law.
Otherwise, by the result of Burns and Medvedev mentioned in the introduction, $G$ has
a normal nilpotent subgroup $N$ such that $G/N$ has finite exponent.
Let $k$ and $e$ be the class of $N$ and the exponent of $G/N$, respectively.
Then the subgroup $G_k^e$ satisfies the law $M_k(x,y)$ and, in particular, the same is
true for the coset $(t_1\ldots t_k)^eA_k^e$.
Now, it follows from Theorem 2.2 that the endomorphism $(t_1^e\ldots t_k^e-1)^k$ is
zero on the abelian group $A_k^e$.
Since $A_k$ is a torsion-free group, $(t_1^e\ldots t_k^e-1)^k$ is also zero as an
endomorphism of $A_k$.
This means that the coset $(t_1\ldots t_k)^eA_k$ satisfies $M_k(x,y)$, which is a
contradiction, according to part (iii) of Lemma~\ref{Gn}. 
\end{proof}

\section*{Acknowledgments}
The authors  are supported by the Spanish Government, grant
MTM2008-06680-C02-02, partly with FEDER funds, and by the
Basque Government, grants IT-252-07 and IT-460-10.
The first author is also supported by a grant of the University
of L'Aquila.


\begin{thebibliography}{9}

\bibitem{baj}
B. Bajorska and O. Macedo\'{n}ska.
`On positive law problems in the class of locally graded groups'.
\textit{Comm. Algebra} \textbf{32} (2004), 1841--1846.

\bibitem{bur}
R.G. Burns and Yu. Medvedev.
`Groups laws implying virtual nilpotence'.
\textit{J. Austral. Math. Soc.} \textbf{74} (2003), 295--312.

\bibitem{cox}
D. Cox, J. Little and D. O'Shea.
\textit{Ideals, Varieties, and Algorithms}
(Springer, 1997), 2nd edn.

\bibitem{fer}
G.A. Fern\'andez-Alcober and P. Shumyatsky.
`Positive laws on word values in residually-$p$ groups'.
Preprint.

\bibitem{gru}
K.W. Gruenberg.
`Residual properties of infinite soluble groups'.
\textit{Proc. London Math. Soc.}(3) \textbf{7} (1957), 29--62.

\bibitem{hup}
B. Huppert.
\textit{Endliche Gruppen I}
(Springer, 1967).

\bibitem{ols}
A.Yu. Olshanskii and A. Storozhev.
`A group variety defined by a semigroup law'.
\textit{J. Austral. Math. Soc. (Series A)} \textbf{60} (1996),
225--259.

\bibitem{rob}
D.J.S. Robinson.
\textit{A Course in the Theory of Groups}
(Springer, 1996), 2nd edn.

\end{thebibliography}
\end{document}